\documentclass[12pt,reqno]{amsart}
\usepackage{amssymb,mathrsfs,accents,color}

\definecolor{deepgreen}{cmyk}{1,0,1,0.5}

\newcommand{\F}{\mathcal{F}}

\newcommand{\cS}{\mathcal{S}}

\newcommand{\R}{\mathbb{R}}
\newcommand{\Z}{\mathbb{Z}}

\newcommand{\al}{\alpha}
\newcommand{\be}{\beta}
\newcommand{\ga}{\gamma}
\newcommand{\de}{\delta}
\newcommand{\e}{\varepsilon}
\newcommand{\fy}{\varphi}

\newcommand{\la}{\lambda}
\newcommand{\te}{\theta}
\newcommand{\s}{\sigma}
\newcommand{\ta}{\tau}
\newcommand{\ka}{\kappa}
\newcommand{\x}{\xi}
\newcommand{\y}{\eta}

\newcommand{\De}{\Delta}

\newcommand{\La}{\Lambda}

\newcommand{\p}{\partial}
\newcommand{\na}{\nabla}
\newcommand{\re}{\mathop{\mathrm{Re}}}

\newcommand{\supp}{\operatorname{supp}}

\newcommand{\lec}{\lesssim}

\newcommand{\I}{\infty}

\newcommand{\ti}{\widetilde}

\newcommand{\ha}{\widehat}
\newcommand{\U}{\underline}
\newcommand{\vect}[1]{\accentset{\rightarrow}{#1}}
\newcommand{\vecn}[1]{\accentset{\rightharpoonup}{#1}}
\newcommand{\veci}[1]{\accentset{\rightharpoondown}{#1}}
\newcommand{\vecd}[1]{\accentset{\dashrightarrow}{#1}}
\newcommand{\LR}[1]{{\langle #1 \rangle}}

\newcommand{\EQ}[1]{\begin{equation}\begin{split} #1 \end{split}\end{equation}}

\newcommand{\reg}{\operatorname{reg}}
\newcommand{\str}{\operatorname{str}}

\newcommand{\Del}[1]{}
\newcommand{\CAS}[1]{\begin{cases} #1 \end{cases}}

\newcommand{\pt}{&}
\newcommand{\pr}{\\ &}
\newcommand{\pq}{\quad}
\newcommand{\pn}{}
\newcommand{\prq}{\\ &\quad}
\newcommand{\prQ}{\\ &\qquad}

\newcommand{\Cstar}{C^\star_{\operatorname{TM}}}

\numberwithin{equation}{section}

\newtheorem{thm}{Theorem}[section]

\newtheorem{lem}[thm]{Lemma}

\theoremstyle{remark}

\begin{document}

\title[Errata]{Errata: Scattering threshold for the focusing nonlinear Klein-Gordon equation}

\author{S.~Ibrahim}
\address{Department of Mathematics and Statistics \\ University of Victoria \\
 PO Box 3060 STN CSC \\ Victoria, BC, V8P 5C3\\ Canada}
\email{ibrahim@math.uvic.ca}
\urladdr{http://www.math.uvic.ca/~ibrahim/}

\author{N.~Masmoudi}
\address{The Courant Institute for Mathematical Sciences \\ New York University} 
\email{masmoudi@courant.nyu.edu} 
\urladdr{http://www.math.nyu.edu/faculty/masmoudi} 

\author{K.~Nakanishi}
\address{Department of Pure and Applied Mathematics \\
Graduate School of Information Science and Technology \\
Osaka University}
\email{nakanishi@ist.osaka-u.ac.jp}

\begin{abstract}
This article resolves some errors in the paper ``Scattering threshold for the focusing nonlinear Klein-Gordon equation", Analysis \& PDE {\bf 4} (2011) no.~3, 405--460. The errors are in the energy-critical cases in two and higher dimensions. 
\end{abstract}

\maketitle

\newcommand{\pp}{\U}

\section{The errors and the missing ingredient}
This article resolves some errors in \cite{SB}. 
One correction affects also \cite{thresol,TM}. 
The section and equation numbers etc.~in \cite{SB} will be underlined for distinction.  
The major errors are the following three: one in \pp{Section 2} for the existence of mass-shifted ground state in the two dimensional energy-critical case, and two in \pp{Section 5} for the nonlinear profile decomposition in the higher dimensional energy-critical case. 
\begin{enumerate}
\item In the proof of \pp{Lemma 2.6}, it is not precluded that the weak limit $Q$ in \pp{(2-67)} is zero. Hence the existence of $Q$ in the case $c\le 1$ is not proved. 
\item In \pp{(5-56)}, we do not have $\|\vect V_n(\ta_n)-\vect V_\I(\ta_n)\|_{L^2_x}\to 0$ when $h_\I=0$, $\ta_\I=\pm\I$ and $\liminf_{n\to\I}|\ta_nh_n^2|>0$. Indeed, assuming that $\ta_nh_n^2\to m\in[-\I,\I]$ after extraction of a subsequence, we have 
\EQ{
 \|\vect V_n(\ta_n)-\vect V_\I(\ta_n)\|_{L^2_x}
 \to \CAS{\|(e^{im/(2|\na|)}-1)\psi\|_{L^2_x} &(|m|<\I), \\
 \sqrt 2\|\psi\|_{L^2_x} &(m=\pm\I).}}
\item In the proof of \pp{Lemma 5.6}, the global bound \pp{(5-96)} does not follow from the uniform bound on finite time intervals, since the required largeness of $n$ depends on the size of the interval $I$. 
\end{enumerate}
(1) is concerned only with a very critical case of exponential nonlinearity in two dimensions $d=2$. More precisely, it is problematic only if 
\EQ{ \label{2dprob}
 0< \limsup_{|u|\to\I}e^{-\ka_0|u|^2}|u|^2f(u)<\I,}
where $\ka_0$ is the exponent in \pp{(1-29)}. (2)--(3) are crucial only in the $H^1$ critical case of higher dimensions $d\ge 3$, with $h_\I=0$: the concentration by scaling in the nonlinear profile, where we need to modify the definition of the nonlinear concentrating waves, and then solve the massless limit problem for NLKG (see Theorem \ref{unif scat} below). 
In the other case, i.e.~with the subcritical or exponential nonlinearity or with $h_\I=1$, we still need to take care of (3), but it is rather superficial change. 

\section{Correction for (1)}
We do not know if \pp{Lemma 2.6} holds true in the very critical case \eqref{2dprob}. So we add the following assumption 
\EQ{ \label{avoid crit}
 \limsup_{|u|\to\I}e^{-\ka_0|u|^2}|u|^2f(u) \in \{0,\I\}}
in \pp{Proposition 1.2(3)} and in \pp{Lemma 2.6}. 
The existence of $Q$ was used in \cite{SB} only to characterize the threshold energy $m$, so the rest of the paper is not affected by it. 

In \cite[(1.24)]{thresol}, the existence of $Q$ is mentioned to characterize the threshold $m^{(c)}$. 
It should be also restricted by \eqref{avoid crit}, but the rest of the paper \cite{thresol} does not really need $Q$. Removing $Q$, \cite[(2.3)]{thresol} should be replaced with 
\EQ{
 m \le H^{(c)}_p(\fy),}
\cite[(2.6)]{thresol} should be replaced with
\EQ{
 m \le J^{(c)}(\la\fy) = H^{(c)}_p(\la\fy) \le H^{(c)}_p(\fy),}
and \cite[(2.7)]{thresol} with 
\EQ{
 \ddot y \pt= (2+p)\|\dot u\|_{L^2}^2 + 2p(H^{(1)}_p(u)-m)
 \pr= (4+\e)\|\dot u\|_{L^2}^2 + (1-c)\e\|u\|_{L^2}^2 + 2p(H^{(c)}_p(u)-m)
 \pr\ge (1+\e/4)\dot y^2/y + (1-c)\e y.}

The existence of $Q$ is also mentioned in \cite[Theorem 5.1]{TM}. 
It should be also restricted by \eqref{avoid crit}. 
The rest of the paper \cite{TM} remains unaffected. 

\medskip

We still need to prove \pp{Lemma 2.6} under the new restriction \eqref{avoid crit}. 
If the limit \eqref{avoid crit} is infinite, then \cite[Theorem 1.5(B)]{TM} implies $\Cstar(F)=\I>1$. 
In this case, the proof of \pp{Lemma 2.6} remains valid. 
If the limit \eqref{avoid crit} is zero, then \cite[Theorem 1.5(B)]{TM} implies $\Cstar(F)<\I$. In this case, we do not argue as in \cite{SB}, but rely on the compactness \cite[Theorem 1.5(C)]{TM}. 
Let $\fy_n\in H^1(\R^2)$ be a normalized maximizing sequence for $\Cstar(F)$, i.e. 
\EQ{
 \|\fy_n\|_{L^2}=1, \pq \ka_0\|\na\fy_n\|_{L^2}^2\le 4\pi, \pq 2F(\fy_n)\to C:=\Cstar(F)\in(0,\I).}
By the standard rearrangement, and the $H^1$ boundedness, we may assume that $\fy_n$ are radially decreasing and $\fy_n\to\exists\fy$ weakly in $H^1(\R^2)$. 
By \cite[Theorem 1.5(C)]{TM}, we have $2F(\fy_n)\to 2F(\fy)=C>0$. In particular, $\fy\not=0$. 
Since $\ka_0\|\na\fy\|_{L^2}^2\le 4\pi$ and $\|\fy\|_{L^2}\le 1$ by the weak convergence, we deduce from the definition of $\Cstar(F)$ that $\|\fy\|_{L^2}=1$ and $\fy$ is a maximizer. 
Hence for a Lagrange multiplier $\mu\ge 0$, 
\EQ{
 f'(\fy)-C\fy = -\mu\De\fy.}
$\mu\not=0$ is obvious by the decay order of $f'$ as $\fy\to 0$. 
Hence $\mu>0$ and so $\ka_0\|\na\fy\|_{L^2}^2=4\pi$, since otherwise we could increase both $F(\fy)$ and $\|\na\fy\|_{L^2}^2$ by the $L^2$ scaling $\fy_{1,-1}^\la$ with $\la>0$, using the $L^2$ super-critical condition \pp{(1-21)}.  
Then $Q(x):=\fy(\mu^{-1/2}x) \in H^2(\R^2)$ satisfies 
\EQ{
 -\De Q + C Q = f'(Q), \pq \ka_0\|\na Q\|_{L^2}^2= 4\pi, 
 \pq 2F(Q)=C\|Q\|_{L^2}^2,}
Hence $J^{(C)}(Q)=\frac 12\|\na Q\|_{L^2}^2 = 2\pi/\ka_0$. 
The rest of the proof of \pp{Lemma 2.6}, namely the proof of $m_{\al,\be}=m_{0,1}=2\pi/\ka_0$ remains valid. 

\section{Correction for (2)-(3)} 
For (2)-(3), we do not have to modify the main results, but need to correct the proof, including the definition of the nonlinear profile decomposition. 
Henceforth, we always assume that $0<h_n\to h_\I$, $(t_n,x_n)\in\R^{1+d}$, and $\ta_n=-t_n/h_n\to\ta_\I\in[-\I,\I]$ are sequences. 
The main problematic case is when the energy concentrates, namely $h_\I=0$, which can happen only in the energy critical case \pp{(1-28)}:
\EQ{
 d \ge 3, \pq f(u)=|u|^{2^\star}/2^\star, \pq 2^\star=2d/(d-2).}

First we modify the vector notation in \pp{(4-1)}. For any real-valued function $a(t,x)$, the complex-valued functions $\vect a,\vecn a,\veci a$ are defined by 
\EQ{ \label{convention}
 \vect a:=(\LR{\na}-i\p_t)a, \pq  \vecn a := (\LR{\na}_n-i\p_t)a, \pq \veci a:=(\LR{\na}_\I-i\p_t)a,} 
where $\LR{\na}_*=\sqrt{h_*^2-\De}$ as in \pp{(5-1)}. Hence $a$ is recovered from either of them by 
\EQ{
 a = \re\LR{\na}^{-1}\vect a= \re\LR{\na}_n^{-1}\ \vecn a= \re\LR{\na}_\I^{-1}\ \veci a.}
Note that $(\veci a, a)$ was denoted by $(\vect a,\ha a)$ in \cite{SB}, but it was confusing. 
Indeed, $u_{(n)}$ in \pp{(5-55)} did not make sense if $h_\I=0$, since $\vect u_{(n)}$ in \pp{(5-54)} was not in the form \pp{(4-1)}. 
So we replace \pp{(5-54)} with 
\EQ{ \label{new conc}
 \vect u_{(n)}=T_n\vecn U_{(n)}((t-t_n)/h_n),}
where $\vecn U_{(n)}$ is defined by 
\EQ{ \label{new prof}
 \vecn V_n:=e^{it\LR{\na}_n}\psi, \pq \vecn U_{(n)}=\vecn V_n - i \int_{\ta_\I}^t e^{i(t-s)\LR{\na}_n}f'(U_{(n)})ds.}
Then $u_{(n)}=h_nT_n U_{(n)}((t-t_n)/h_n)$ is a solution of NLKG satisfying 
\EQ{
 \lim_{t\to\ta_\I}\|(\vect u_{(n)}-\vect v_n)(t h_n+t_n)\|_{L^2_x}= 0.}
In other words, we keep NLKG in defining the profiles, even if $h_\I=0$. 
Note that if $h_\I=1$ then $\vecn U_{(n)}=\vect U_\I$ and so $u_{(n)}$ is unchanged. 


By the change of \pp{(5-54)} to \eqref{new conc}, the problematic \pp{(5-56)} is replaced with 
\EQ{ \label{new 5-56}
 \|\vect u_n(0)-\vect u_{(n)}(0)\|_{L^2_x}
 \pn= \|\int_{\ta_\I h_n+t_n}^{0\ (=\ta_n h_n+t_n)}e^{-is\LR{\na}}f'(u_{(n)})ds\|_{L^2_x} \to 0.}
In order to prove the last limit, as well as the global Strichartz approximation for (3), we need the convergence in the massless limit of the $H^1$ critical NLKG:
\begin{thm} \label{unif scat} 
Assume \pp{(1-28)} and $h_\I=0$. Let $\veci U_\I$ be the solution of 
\EQ{
 \veci V_\I:=e^{it|\na|}\psi, \pq \veci U_\I=\veci V_\I-i\int_{\ta_\I}^t e^{i(t-s)|\na|}f'(U_\I)ds.}
Let $\vecn U_{(n)}$ be the solution of \eqref{new prof} and $\vect u_{(n)}(t):=T_n\vecn U_{(n)}((t-t_n)/h_n)$. 
Suppose that $U_\I \in[W]^\bullet_2(J)$ for some interval $J$ whose closure in $[-\I,\I]$ contains $\ta_\I$. 
Then for any bounded subinterval $I\subset J$, we have, as $n\to\I$, 
\EQ{
 \pt\|\vecn U_{(n)}-\veci U_\I\|_{L^\I_{t\in I}L^2_x}
  +\|U_{(n)}-U_\I\|_{([W]_2^\bullet\cap[M]_0)(J)} 
  +\|u_{(n)}\|_{[W]_0(J)}\to 0, 
 \pr \|u_{(n)}\|_{([W]_2\cap[M]_0)(h_nJ+t_n)} \sim \|U_\I\|_{([W]_2^\bullet\cap[M]_0)(J)}+o(1).}
\end{thm}

Postponing the proof of the above theorem to the next section, we continue to correct \pp{Section 5}. \eqref{new 5-56} in the case of $h_\I=0$ follows from the above estimate and $\ta_n\to\ta_\I$ via Strichartz: 
\EQ{
 \pt\|\int_{\ta_\I h_n+t_n}^0 e^{-is\LR{\na}}f'(u_{(n)})ds\|_{L^2_x}
 \pn\lec \|f'(u_{(n)})\|_{[W^{*(1)}]_2(I_n)} 
 \prq\lec \|u_{(n)}\|_{([W]_2\cap[M]_0)(I_n)}^{2^\star-1}
 \pn\lec \|U_\I\|_{[W]_2^\bullet\cap[M]_0(J_n)}^{2^\star-1}+o(1)=o(1),}
where $I_n:=(0,\ta_\I h_n+t_n)\cup(\ta_\I h_n+t_n,0)$ 
and $J_n:=(\ta_n,\ta_\I)\cup(\ta_\I,\ta_n)$. 

We modify the definition of $ST$ in \pp{(5-59)--(5-60)} in the $\dot H^1$ critical case \pp{(1-28)} to 
\EQ{
 \pt ST=[W]_2, \pq ST^*=[W^{*(1)}]_2+L^1_tL^2_x,
 \pq ST^\diamondsuit_\I:=\CAS{[W]_2 &(h^\diamondsuit_\I=1), \\ [W]_2^\bullet &(h^\diamondsuit_\I=0).}}
Indeed, $[K]_2$ and $[K^{*(1)}]_2$ norms are not needed in the $\dot H^1$ critical case. 
Then we simply discard the estimates \pp{(5-61)--(5-62)}. 

Next we reprove \pp{Lemma 5.5}, extending it to unbounded intervals $I$. 
The above theorem implies that we can replace \pp{(5-64)} with the stronger\footnote{Recall that $\ha U^j_\I$ in \cite{SB} is denoted by $U^j_\I$ in this errata according to \eqref{convention}.} 
\EQ{
 \limsup_{n\to\I}\|u^j_{(n)}\|_{ST(\R)} \lec \|U^j_\I\|_{ST^j_\I(\R)},}
if $h^j_\I=0$, while it is trivial if $h^j_\I=1$. 
The proof of \pp{(5-65)} for $h^j_\I=1$ did not use the boundedness of $I$, so we may assume that all $h^j_\I$ are $0$. Then the above theorem implies that $\|u^{<k}_{(n)}\|_{[W]_0(\R)}\to 0$ as $n\to\I$, so it suffices to estimate the homogeneous norm $[W]_2^\bullet(\R)$. We have 
\EQ{
 \|u^{<k}_{(n)}\|_{[W]_2^\bullet(\R)} \sim \sum_{l=1}^d\|\sum_{j<k}\check u^{j,l}_{n,m}\|_{L^p_t \ell^2_{m\in\Z}L^q_x}}
with $(1/p,1/q,s)=W$ and 
\EQ{
 \check u^{j,l}_{n,m}:=2^{sm}\de^l_m h^j_n T^j_n U^j_{(n)}((t-t^j_n)/h^j_n).}
Defining $\check u^{j,l}_{n,m,R}$ by \pp{(5-77)}, we have 
\EQ{
 \|\check u^{j,l}_{n,m}-\check u^{j,l}_{n,m,R}\|_{L^p_t\ell^2_mL^q_x}
 \lec \|2^{sm}\de^l_m U^j_{(n)}\|_{L^p_t\ell^2_mL^q_x(|t|+|m|+|x|>R)} \to 0,\pq(R\to\I)}
which is still uniform in $n$, since by the above theorem $U^j_{(n)}$ is approximated by $U^j_\I$ in $[W]_2^\bullet(\R)$, which is equivalent to the last norm without the restriction by $R$. 
Thus we obtain \pp{(5-65)} by the disjoint support property for large $n$. 

According to the change of $u^j_{(n)}$, we replace the nonlinear decomposition \pp{(5-66)} with a simpler form: 
\EQ{ \label{nonlin orth}
 \lim_{n\to\I}\|f'(u^{<k}_{(n)})-\sum_{j<k}f'(u^j_{(n)})\|_{ST^*(I)}=0,}
which is the same as \pp{(5-66)} if $h^j_\I=1$. 
In that case, however, we used that $I$ was bounded in \pp{(5-82)}. We replace it with an interpolation between \pp{(4-84)} and 
\EQ{ \label{fest in Z}
 \|f'_S(u)\|_{[((1-\te_0)K+\te_0 W)^{*(1)}]_2(I)} \lec \|u\|_{[K]_2(I)}\|u\|_{[K]_0(I)}^{p_1} \lec \|u\|_{[K]_2(I)}^{p_1+1},}
where we can choose some $\te_0\in(0,1)$ since $p_1>4/d$ (and choosing $p_1$ close enough to $4/d$ if necessary). 
Since $Z:=((1-\te_0)K+\te_0 W)^{*(1)}$ is an interior dual-admissible exponent, we can find some $\te_1\in(0,1)$ such that $\te_1Y+(1-\te_1)Z$ is also a dual-admissible exponent. Interpolating \eqref{fest in Z} with \pp{(4-84)}, we have 
\EQ{
 \|f'_S(u)-f'_S(v)\|_{[\te_1Y+(1-\te_1)Z]_2(I)}
 \lec \|(u,v)\|_{[K]_2(I)\cap[Q]_{2p_1}(I)}^{p_1+1-\te_1}\|u-v\|_{[P]_2(I)}^{\te_1}.}
Thus we obtain \pp{(5-66)} on any subset $I$ in the subcritical/exponential cases. 
In the $\dot H^1$ critical case \pp{(1-28)}, we discard $u_{\LR{n}}^j$ in \pp{(5-85)} and prove \eqref{nonlin orth} directly, putting 
\EQ{
 U^j_{n,R}(t,x)\pt:=\chi_R(t,x)U_{(n)}^j(t,x)
 \prQ\times \prod\{(1-\chi_{h_n^{j,l}R})(t-t_n^{j,l},x-x_n^{j,l})\mid 1\le l<k,\ h_n^lR<h_n^j\}.}
It is still uniformly bounded in $([H]_2^\bullet\cap[W]_2^\bullet)(\R)$, and $U^j_{n,R}-\chi_RU_{(n)}^j\to 0$ in $[M]_0(\R)$ as $n\to\I$, thanks to the above theorem, as well as in $[L]_0$, and also $\chi_RU_{(n)}^j\to U_{(n)}^j$ as $R\to\I$. 
Hence we may replace $u_{(n)}^j$ in \eqref{nonlin orth} by $u_{(n),R}^j:=h_n^jT_n^jU^j_{n,R}((t-t_n^j)/h_n^j)$, using \pp{(4-62)} for $d\le 5$, and a similar interpolation argument as above for $d\ge 6$, see \eqref{est in Y}--\eqref{est with O} below. 
Then we obtain \eqref{nonlin orth} by the disjoint support property, in the same way as \pp{(5-94)}.

With the above corrections, now we reprove \pp{Lemma 5.6}. 
First, \pp{(5-100)} holds for any subset $I\subset\R$, by the above improvement of \pp{Lemma 5.5}. 
Now, thanks to the change of $u^j_{(n)}$,  \pp{(5-101)} is simplified to 
\EQ{
 eq(u^{<k}_{(n)}) = f'(u^{<k}_{(n)})-\sum_{j<k}f'(u^j_{(n)}),}
which is vanishing by \eqref{nonlin orth}.  
Hence we obtain \pp{(5-103)}. 
We also obtain \pp{(5-104)} on $\R$ by the same nonlinear estimates as we used above. 
Then applying \pp{Lemma 4.5} on $\R$, we obtain the desired \pp{Lemma 5.6}. 

\pp{Section 6} is almost unchanged, except for the obvious modification in \pp{(6-6)} due to the change of $u_{(n)}$, namely
\EQ{
 \vect u^j_{(n)}=T_n^j \vecn U^j_{(n)}((t-t^j_n)/h^j_n),}
and the notational change in \U{(6-7)--(6-9)} from $(\vec U^0_\I,\ha U^0_\I)$ to $(\veci U^0_\I,U^0_\I)$ due to \eqref{convention}. 
Since the case $h_\I=0$ is eliminated in the proof of \pp{Lemma 6.1}, the errors (2)-(3) do not affect the rest of the paper.

\section{Massless limit of scattering for the critical NLKG}
It remains to prove Theorem \ref{unif scat}. Throughout this section, we assume  \pp{(1-28)}. 
The main idea is to decompose the time interval into a bounded subinterval and neighborhoods of $\pm\I$. 
On the bounded part, we have strong convergence in the massless limit. 
In the neighborhoods of $t=\pm\I$, we do not have strong convergence, but the Strichartz norms are uniformly controlled via the asymptotic free profiles. 

The first ingredient concerns the uniform Strichartz bound for free waves.
\begin{lem} \label{lem:unif St}
Let $\vect v_n=e^{it\LR{\na}}T_n\psi$, $h_\I=0$, $\veci V_\I=e^{it|\na|}\psi$, and let $Z\in[0,1/2]\times[0,1/2)\times[0,1)$ satisfy $\reg^0(Z)=1$ and $\str^0(Z)\le 0$, namely a wave-admissible Strichartz exponent except for the energy norm. Then we have 
\EQ{
 \limsup_{n\to\I}\|v_n\|_{[Z]_2(0,\I)}\lec \|V_\I\|_{[Z]_2^\bullet(0,\I)},
 \pq \lim_{n\to\I}\|P_{<1}v_n\|_{[Z]_2(0,\I)}=0, }
where $P_{<a}$ denotes the smooth cut-off for the Fourier region $|\x|<2a$ defined by $P_{<a}\fy = a^d \La_0(ax)*\fy$, 
with $\La_0\in\cS(\R^d)$ in the proof of \pp{Lemma 5.1}. 
If $Z_3=0$, then we have also $\|v_n\|_{[Z]_0(0,\I)}\to \|V_\I\|_{[Z]_0(0,\I)}$. 
\end{lem}
\begin{proof}
Let $\vect v_n(t)=T_n\vecn V_n(t/h_n)$. 
The Strichartz estimate for the Klein-Gordon and the wave equations 
\EQ{
 \|v_n\|_{[Z]_2(0,\I)} \lec \|T_n\psi\|_{L^2} = \|\psi\|_{L^2},
 \pq \|V_\I\|_{[Z]_2^\bullet(0,\I)}\lec \|\psi\|_{L^2}}
implies that it suffices to consider $\psi$ in a dense subset of $L^2(\R^d)$. 
Hence we may assume that $\F\psi$ is $C^\I$ with a compact $\supp\F\psi\not\ni 0$. Since $0<\LR{\x}_n-\LR{\x}_\I\le h_n^2/|\x|$, 
\EQ{
 |(e^{it\LR{\x}_n}\LR{\x}_n^{-1}-e^{it|\x|}|\x|^{-1})| \lec |t|h_n^2|\x|^{-2}+h_n^2|\x|^{-3},}
and so, under the above assumption on $\psi$, for any $s\in\R$, and any sequence $S_n>0$, 
\EQ{ 
  \|V_n-V_\I\|_{L^\I(0,S_n;H^s)} \le \LR{S_n}h_n^2C(s,\psi).}
Hence by Sobolev in $x$ and H\"older in $t$, 
\EQ{
 \|V_n-V_\I\|_{([Z]_2^\bullet\cap[Z]_0)(0,S_n)} \le \LR{S_n}^{1+Z_1}h_n^2C(s,\psi).}
We deduce that if $S_n\to\I$ and $S_n^{1+Z_1}h_n^2 \to 0$, 
then using the (approximate) scale invariance of $[Z]_2^\bullet$, 
\EQ{
 \pt\|v_n\|_{[Z]_2(0,h_nS_n)} \sim \|v_n\|_{[Z]_2^\bullet(0,h_nS_n)} + \|P_{<1}v_n\|_{[Z]_0(0,h_nS_n)},
 \prq \|v_n\|_{[Z]_2^\bullet(0,h_nS_n)}\sim \|V_n\|_{[Z]_2^\bullet(0,S_n)} \to \|V_\I\|_{[Z]_2^\bullet(0,\I)},
 \prq \|P_{<1}v_n\|_{[Z]_0(0,h_nS_n)} \sim \|h_n^{Z_3}P_{<h_n}V_n\|_{[Z]_0(0,S_n)} \to 0,}
and similarly if $Z_3=0$, 
$\|v_n\|_{[Z]_0(0,h_nS_n)}=\|V_n\|_{[Z]_0(0,S_n)} \to \|V_\I\|_{[Z]_0(0,\I)}$. 

Next, the dispersive decay of wave-type for the Klein-Gordon equation 
\EQ{
 \|e^{it\LR{\na}}\fy\|_{B^0_{q,2}} \lec |t|^{-(d-1)\al}\|\fy\|_{B^{s}_{q',2}}
 \pq \al:=\frac{1}{2}-\frac{1}{q}\in[0,1/2],\pq s:=(d+1)\al,}
together with the embedding $L^{q'}\subset B^0_{q',2}$ implies that 
\EQ{
 \|v_n(t)\|_{B^{\s}_{q,2}} \pt\lec |t|^{-(d-1)\al}\|\LR{\na}^{\s+s-1}T_n\psi\|_{L^{q'}}
 \pr=|t|^{-(d-1)\al}h_n^{1-\al-\s}\|\LR{\na}_n^{\s+s-1}\psi\|_{L^{q'}},}
and so, putting $\al=1/2-Z_2$, 
\EQ{
 \|v_n\|_{[Z]_2(h_nS_n,\I)} \pt\le C(\psi)h_n^{1-\al-Z_3}\|t^{-(d-1)\al}\|_{L^{1/Z_1}_t(h_nS_n,\I)}
 \pr\sim C(\psi)h_n^{1-\al-Z_3}(h_nS_n)^{Z_1-(d-1)\al}
 = C(\psi) S_n^{\al-1+Z_3} \to 0}
where we used that $\reg^0(Z)=Z_3-Z_1+d\al=1$ in the last identity, and 
\EQ{
 \al-1+Z_3=\reg^0(Z)+\str^0(Z)-1-Z_1 < 0}
in taking the limit. 
Note that the above exponent is zero at the energy space $Z=(0,1/2,1)$, which is excluded by the assumption. 
The estimate in $[Z]_0(h_nS_n,\I)$ for $Z_3=0$ is done in the same way. 
Combining them with the above estimates on $(0,h_nS_n)$ leads to the conclusion via the density argument. 
\end{proof}
The second ingredient is convergence or propagation of small disturbance on finite intervals, which is uniformly controlled by the Strichartz norm of $U_\I$. 
\begin{lem} \label{lem:conv}
For any $0<M,\e<\I$, there exists $\de=\de(\e,M)\in(0,1)$ with the following property. 
Let $h_\I=0$ and let $U_\I$ be a solution of NLW on some interval $J$ satisfying 
$\|U_\I\|_{([H]_2^\bullet\cap[W]_2^\bullet)(J)}\le M$. 
Then for any bounded subinterval $I\subset J$ with $0\in I$ and any $\fy_n\in L^2(\R^d)$ with $\|\fy_n\|_{L^2}<\de$, the unique solution $U_n$ of 
\EQ{
 (\p_t^2-\De+h_n^2)U_n=f'(U_n), \pq \vecn U_n(0)=\veci U_\I(0)+\fy_n}
exists on $I$ for large $n$, satisfying
\EQ{ \label{local conv}
 \pt \|\vecn U_n-\veci U_\I\|_{L^\I_t L^2_x(I)} + \|U_n - U_\I\|_{([W]_2^\bullet\cap[M]_0)(I)}<\e,} 
and $\|h_nT_nU_n((t-t_n)/h_n)\|_{[W]_0(h_nI+t_n)}\lec\de$ for large $n$. 
\end{lem}
\begin{proof}
We give the detail only in the harder case $d\ge 6$, where we need the exotic Strichartz norms. Let $\ga_n:=U_n- U_\I$ and $\vecd\ga_n:=\vecn U_n-\veci U_\I$, then 
\EQ{ \label{eq gan}
 (\p_t^2-\De)\ga_n=f'(U_\I+\ga_n)-f'(U_\I)-h_n^2U_n.}
Remark however that $\vecd\ga_n$ is not written only by $\ga_n$. 
It suffices to prove the following 

\medskip

{\bf Claim.} There exist constants $\te\in(0,1)$ and $C>1$ such that if 
\EQ{ \label{small ST}
 \|U_\I\|_{([W]_2^\bullet\cap[\ti M]_{2p}^\bullet)(0,S)} \le \y, \pq \|\vecd\ga_n(0)\|_{L^2} \ll 1}
for some $0<S<\I$ and $0<\y\ll 1$, where $p=2^\star-2=4/(d-2)$, then 
\EQ{ \label{claim}
 \pt\|\vecd\ga_n\|_{L^\I_t(0,S;L^2_x)} + \|\ga_n\|_{[W]_2^\bullet(0,S)} 
 \pn\le C[\|\vecd\ga_n(0)\|_{L^2} + \|\vecd\ga_n(0)\|_{L^2}^{\te}\y^{(p+1)(1-\te)}].}
\begin{proof}[Proof of the claim] 
The exotic Strichartz estimate for the wave equation yields on the time interval $(0,S)$ 
\EQ{
 \pt\|\ga_n\|_{[\ti N]_2^\bullet} \lec \|\veci\ga_n(0)\|_{L^2}+\|f'(U_\I+\ga_n)-f'(U_\I)\|_{[Y]_2} + \|h_n^2U_n\|_{L^1_tL^2_x},}
while the nonlinear estimate in the Besov space yields
\EQ{ \label{est in Y}
 \pt\|f'(U_\I+\ga_n)-f'(U_\I)\|_{[Y]_2} 
 \pr\lec \|(U_\I,\ga_n)\|_{[M]_0}^p\|\ga_n\|_{[\ti N]_2^\bullet}+\|(U_\I,\ga_n)\|_{[\ti M]_{2p}^\bullet}^p\|\ga_n\|_{[N]_0},}
and we have $\|\veci\ga_n(0)\|_{L^2} \lec \|\vecd\ga_n(0)\|_{L^2}+o(1)$. The $L^1_tL^2_x$ norm is estimated by 
\EQ{
 \|h_n^2U_n\|_{L^1_tL^2_x} \le \|h_n\vecn U_n\|_{L^1_tL^2_x} \le 
h_nS\|\vecd\ga_n+\veci U_\I\|_{L^\I_tL^2_x}.} 

Define $\U{W},O\in[0,1/2]^3$ by 
\EQ{
 \pt\U{W}:=W-\frac{1}{2}(0,1/d,1)=(\frac{d-1}{2(d+1)},\frac{d^2-2d-1}{2d(d+1)},0), 
 \pr O:=W+p\U{W}=(\frac{(d+2)(d-1)}{2(d+1)(d-2)},\frac{d^3+d^2-6d-4}{2(d-2)d(d+1)},1/2).}
Then $O$ is an interior dual exponent of the standard Strichartz, and so, there is small $\te\in(0,1)$ such that $\te Y+(1-\te)O$ is also a dual exponent. 
Hence the standard Strichartz yields for any wave-admissible exponent $Z$, 
\EQ{ \label{est with O}
 \pt \|\ga_n\|_{[Z]_2^\bullet}+\|\veci\ga_n\|_{L^\I_tL^2_x} 
 \prQ\lec \|\veci\ga_n(0)\|_{L^2} + \|f'(U_\I+\ga_n)-f'(U_\I)\|_{[\te Y+(1-\te)O]^\bullet_2}+\|h_n^2U_n\|_{L^1_tL^2_x},}
where the nonlinear part is already estimated in $[Y]^\bullet_2$, while  
\EQ{
 \|f'(U_\I+\ga_n)\|_{[O]^\bullet_2}+\|f'(U_\I)\|_{[O]^\bullet_2} \lec \y^{p+1} + \|\ga_n\|_{[W]^\bullet_2}^{p+1}.}
Hence we have 
\EQ{
 \pt\|\ga_n\|_{[\ti N]^\bullet_2} \lec \|\veci\ga_n(0)\|_{L^2} + A + B,
 \pr\|\ga_n\|_{[W]_2^\bullet\cap[\ti M]_{2p}^\bullet}+\|\veci\ga_n\|_{L^\I_tL^2_x}  
 \pn\lec \|\veci\ga_n(0)\|_{L^2}+A^\te(\y + \|\ga_n\|_{[W]^\bullet_2})^{(1-\te)(p+1)}+B,
 \pr A \lec (\y+\|\ga_n\|_{[\ti M]^\bullet_{2p}})^p\|\ga_n\|_{[\ti N]^\bullet_2},
 \pq B \lec Sh_n\|\vecd\ga_n\|_{L^\I_tL^2_x}+o(1).}
Assuming that $\|\ga_n\|_{[\ti M]_{2p}^\bullet}\ll 1$ and that $\|\vecd\ga_n\|_{L^\I_tL^2_x}$ is bounded in $n$, we deduce from the above estimates that 
\EQ{ \label{boot}
 \pt A \ll \|\ga_n\|_{[\ti N]^\bullet_2} \lec \|\veci\ga_n(0)\|_{L^2}+o(1), \pq B=o(1),
 \pr \|\ga_n\|_{[W]_2^\bullet\cap[\ti M]_{2p}^\bullet}+\|\veci\ga_n\|_{L^\I_tL^2_x} \lec \|\veci\ga_n(0)\|_{L^2} + \|\veci\ga_n(0)\|_{L^2}^\te \y^{(1-\te)(p+1)}+o(1).}

It remains to prove the uniform bound on $\|\vecd\ga_n\|_{L^\I_tL^2_x}$. 
Let $V_\I,V_n,v_n$ be the free solutions defined by 
\EQ{
 \veci V_\I:=e^{it|\na|}\veci U_\I(0), \pq 
 \vecn V_n:=e^{it\LR{\na}_n}\vecn U_n(0),
 \pq \vec v_n=T_n\vecn V_n(t/h_n). }
For any $0<R_n\to 0$ such that $h_n/R_n\to 0$, we have 
\EQ{ \label{tiga hf}
 \|\F\vecd\ga_n\|_{L^\I(0,S;L^2(|\x|>R_n))} \lec \|\veci\ga_n\|_{L^\I(0,S;L^2_x)}+o(1).}
For the lower frequency, we have by the energy inequality, H\"older and Sobolev, 
\EQ{ \label{low freq diff}
 \|\vecn U_n-\vecn V_n\|_{L^\I_t \dot H^{-1}_x(0,S)} \pt\lec \|f'(U_n)\|_{L^1_t\dot H^{-1}_x(0,S)}
 \pn\lec S\|U_n\|_{L^\I_t \dot H^1_x(0,S)}^{p+1} 
 \pr\lec S(\|\veci U_\I\|_{L^\I_t L^2_x(0,S)}+\|\veci\ga_n\|_{L^\I_tL^2_x(0,S)})^{p+1},}
and similarly, $\|\veci U_\I-\veci V_\I\|_{L^\I_t \dot H^{-1}_x(0,S)} \lec S\|\veci U_\I\|_{L^\I_tL^2_x}^{p+1}$. 
Since $|\LR{\x}_n-\LR{\x}_\I| \le h_n$, we have also 
$\|\vecn V_n(t)-\veci V_\I(t)\|_{L^2_x}\lec |t|h_n\|\veci U_\I(0)\|_{L^2}+\de$. 
Hence 
\EQ{
 \|\F\vecd\ga_n\|_{L^\I(0,S;L^2(|\x|<R_n))}
 \pt\le R_n\|\vecn U_n-\vecn V_n\|_{L^\I_t \dot H^{-1}_x(0,S)}
   +\|\vecn V_n-\veci V_\I\|_{L^\I_tL^2_x(0,S)}  
  \prq+R_n\|\veci V_\I-\veci U_\I\|_{L^\I_t \dot H^{-1}_x(0,S)} 
 \pr\lec o(1)S\|\veci\ga_n\|_{L^\I_tL^2_x(0,S)}^{p+1}+\de+o(1)}
Adding it to \eqref{tiga hf}, we obtain 
\EQ{
 \|\vecd\ga_n\|_{L^\I_tL^2_x(0,S)}
 \lec \|\veci\ga_n\|_{L^\I_tL^2_x(0,S)} + o(1)S\|\veci\ga_n\|_{L^\I_tL^2_x(0,S)}^{p+1} + \de + o(1).}
Combining it with the above estimates \eqref{boot}, we deduce that both $\vecd\ga_n$ and $\veci \ga_n$ are bounded in $L^\I_tL^2_x(0,S)$. 
\end{proof}

To prove \eqref{local conv} from the above claim, we decompose $I$ into subintervals $I_j$, such that $\|U_\I\|_{([W]_2^\bullet\cap[\ti M]_{2p}^\bullet)(I_j)}\le \y$ for each $j$. 
Then applying the above claim iteratively to the subintervals for small $\de>0$ yields \eqref{local conv}, where the bound on $[M]_0$ is derived by interpolation and Sobolev embedding of $[H]_2^\bullet$ and $[W]_2^\bullet$. 

For the estimate in $[W]_0$, we have by scaling 
\EQ{ \label{un in W0}
 \pt\|h_nT_nU_n((t-t_n)/h_n)\|_{[W]_0(h_nI+t_n)}
   \sim h_n^{1/2}\|U_n\|_{[W]_0(I)}
 \prq \lec h_n^{1/2}\|U_n\|_{[W]_2^\bullet(I)} + \|P_{<1}v_n\|_{[W]_0(I)} + h_n^{1/2} \|P_{<h_n}(U_n-V_n)\|_{[W]_0(I)},}
where $\vecn V_n:=e^{it\LR{\na}_n}\vecn U_n(0)$ and $\vec v_n=T_n\vecn V_n(t/h_n)$. 
The first term on the right is vanishing since $\|U_n\|_{[W]_2^\bullet(I)}$ is bounded as shown above. 
The second term is $O(\de)$ by Lemma \ref{lem:unif St}. 
The third term is bounded, using Sobolev, H\"older and the same estimate as in \eqref{low freq diff}, by
\EQ{
 \pt |I|^{W_1}h_n^{1/2+d(1/2-W_2)}\|U_n-V_n\|_{L^\I_tL^2_x(I)}
 \prq\lec (|I|h_n)^{3/2-1/(d+1)}(\|\veci U_\I\|_{L^\I_tL^2_x(I)}+\e)^{p+1}=o(1),}
hence \eqref{un in W0} is $O(\de)$ for large $n$. 
This concludes the proof of the lemma for $d\ge 6$. 

The case $d\le 5$ is the same, but the nonlinear estimate is much simpler. 
In \eqref{small ST}, $[\ti M]_{2p}^\bullet$ is replaced with $[M]_0$, and by the standard Strichartz, we have 
\EQ{
 \pt\|\ga_n\|_{[W]_2^\bullet\cap[M]_0} + \|\veci \ga_n\|_{L^\I_tL^2_x}
 \prq\lec \|\veci \ga_n(0)\|_{L^2} + \|f'(U_\I+\ga_n)-f'(U_\I)\|_{[W^{*(1)}]_2^\bullet} + \|h_n^2U_n\|_{L^1_tL^2_x},}
and  
\EQ{ 
 \|f'(U_\I+\ga_n)-f'(U_\I)\|_{[W^{*(1)}]_2^\bullet} \pt\lec \|(U_\I,\ga_n)\|_{[W]_2^\bullet\cap[M]_0}^p\|\ga_n\|_{[W]_2^\bullet\cap[M]_0}
 \pr\lec (\y+\|\ga_n\|_{[W]_2^\bullet\cap[M]_0})^p\|\ga_n\|_{[W]_2^\bullet\cap[M]_0}.}
Then estimating $\|h_n^2U_n\|_{L^1_tL^2_x(0,S)}$ in the same way as for $d\ge 6$, we obtain \eqref{claim} without the last term. \eqref{un in W0} is the same as above. 
\end{proof}

\begin{proof}[Proof of Theorem \ref{unif scat}]
Let $v_n,V_n,V_\I$ be the free solutions defined by
\EQ{
 \vecn V_n=e^{it\LR{\na}_n}\psi, \pq \veci V_\I=e^{it|\na|}\psi, \pq \vec v_n=T_nV_n((t-t_n)/h_n),}
and 
\EQ{
 M:=\|U_\I\|_{[W]_2^\bullet(J)}.} 

First consider the case $\ta_\I=\I$. 
Let $0<\e<1$ and choose $S>0$ so large that 
\EQ{
 \de_0:=\|V_\I\|_{([W]_2^\bullet\cap[M]_0)(S,\I)} \le \de(\e,M),}
where $\de(\cdot,\cdot)$ is given by Lemma \ref{lem:conv}. 
Then Lemma \ref{lem:unif St} implies that 
\EQ{
 \|v_n\|_{([W]_2\cap[M]_0)(h_nS+t_n,\I)} \lec \de_0}
for large $n$. If $\de_0\ll 1$, then the standard scattering argument for NLKG using the Strichartz norms implies that $u_{(n)}$ exists on $(h_nS+t_n,\I)$, satisfying
\EQ{
 \|\vect u_{(n)}-\vect v_n\|_{L^\I_t L^2_x(h_nS+t_n,\I)}+ \|u_{(n)}-v_n\|_{([W]_2\cap[M]_0)(h_nS+t_n,\I)} \lec \de_0^{2^\star-1} \ll \de_0,} 
and also for NLW 
\EQ{
 \|\veci U_\I-\veci V_\I\|_{L^\I_t L^2_x(S,\I)}+ \|U_\I-V_\I\|_{([W]_2^\bullet\cap[M]_0)(S,\I)} \lec \de_0^{2^\star-1} \ll \de_0.}
Thus we obtain 
\EQ{
 \|u_{(n)}\|_{([W]_2\cap[M]_0)(h_nS+t_n,\I)} 
 \pn\lec \|V_\I\|_{([W]_2^\bullet\cap[M]_0)(S,\I)}\sim \|U_\I\|_{([W]_2^\bullet\cap[M]_0)(S,\I)},}
and, for large $n$, 
\EQ{
 \|\vecn U_{(n)}(S)-\vecn V_n(S)\|_{L^2_x} + \|\vecn V_n(S)-\veci V_\I(S)\|_{L^2_x}+ \|\veci V_\I(S)-\veci U_\I(S)\|_{L^2_x} \ll \de_0. } 

The next step is to go from $S$ to the negative time direction. 
If $J$ is bounded from below, then let $S':=\inf J$. 
Otherwise, choose $S'<S$ so that 
\EQ{
 \|U_{\I}\|_{([W]_2^\bullet\cap[M]_0)(-\I,S')}< \e.} 
Applying Lemma \ref{lem:conv} to $U_\I$ and $U_{(n)}$ backward in time from $t=S$, we obtain 
\EQ{
 \pt \|\vecn U_{(n)}-\veci U_{\I}\|_{L^\I_tL^2_x(S',S)}+\|U_{(n)}-U_{\I}\|_{([W]_2^\bullet\cap[M]_0)(S',S)}<\e,} 
and $\|u_{(n)}\|_{[W]_0(h_nS'+t_n,h_nS+t_n)}\lec\de_0$ for large $n$. 

If $J$ is unbounded from below, we have still to go from $S'$ to $-\I$. 
The standard argument for small data scattering of NLW for $t\to-\I$ implies that 
\EQ{
 \pt\|\re|\na|^{-1}e^{it|\na|}\veci U_{\I}(S')\|_{([W]_2^\bullet\cap[M]_0)(-\I,0)}  \pn\sim \|U_{\I}\|_{([W]_2^\bullet\cap[M]_0)(-\I,S')} < \e.}
Then Lemma \ref{lem:unif St} applied backward in $t$ implies for large $n$ 
\EQ{
 \|\re\LR{\na}^{-1}e^{it\LR{\na}}T_n\veci U_{\I}(S')\|_{([W]_2\cap[M]_0)(-\I,0)} \lec \e.}
Let $w_n$ be the solution of NLKG with $\vect w_n(0)=T_n\vecn U_{(n)}(S')$. 
Then the above estimate together with $\|\vecn U_{(n)}(S')-\veci U_{\I}(S')\|_{L^2_x}<\e$ and the scattering for NLKG implies 
\EQ{
 \|w_n\|_{([W]_2\cap[M]_0)(-\I,0)} \lec \e.}
Since $w_n=h_nT_nU_{(n)}(t/h_n+S')=u_{(n)}(t+h_nS'+t_n)$, we deduce that 
\EQ{
 \pt\|U_{(n)}\|_{([W]_2^\bullet\cap[M]_0)(-\I,S')} \sim \|u_{(n)}\|_{([W]_2^\bullet\cap[M]_0)(-\I,h_nS'+t_n)}
 \prq\lec \|u_{(n)}\|_{([W]_2\cap[M]_0)(-\I,h_nS'+t_n)} = \|w_n\|_{([W]_2\cap[M]_0)(-\I,0)}  \lec \e.} 

Thus we obtain, in the case $\ta_\I=\I$, 
\EQ{
 \|U_{(n)}-U_\I\|_{([W]_2^\bullet\cap[M]_0)(J)}+\|u_n\|_{[W]_0(h_nJ+t_n)} \lec \e+\de_0}
for large $n$. Since $\e$ and $\de_0$ can be chosen as small as we wish, it implies 
\EQ{ \label{global St conv}
 \lim_{n\to\I} \|U_{(n)}-U_\I\|_{([W]_2^\bullet\cap[M]_0)(J)}+\|u_n\|_{[W]_0(h_nJ+t_n)}=0,}
and by scaling, 
\EQ{
 \|u_{(n)}\|_{([W]_2\cap [M]_0)(h_n J+t_n)} \pt\sim \|U_\I\|_{([W]_2^\bullet\cap[M]_0)(J)}+\|u_{(n)}\|_{[W]_0(h_n J+t_n)}
 \pr=\|U_\I\|_{([W]_2^\bullet\cap[M]_0)(J)}+o(1). }
 
Since $S\to\I$ and $S'\to\inf J$ as $\e,\de\to+0$, we also obtain 
\EQ{ \label{local H conv}
 \lim_{n\to\I}\|\vecn U_{(n)}-\veci U_\I\|_{L^\I_tL^2_x(I)}=0,}
for any finite subinterval $I$. 
The case $\ta_\I=-\I$ is the same by the time symmetry. 

If $\ta_\I\in\R$, then $\|\vecn U_{(n)}(\ta_\I)-\veci U_\I(\ta_\I)\|_{L^2_x}\to 0$. Hence the same argument as we used above to go from $S$ to $-\I$ yields 
\EQ{
 0 \pt=\lim_{n\to\I}\|\vecn U_{(n)}-\veci U_\I\|_{L^\I_tL^2_x(S',\ta_\I)}
  \pn=\lim_{n\to\I}\|U_{(n)}- U_\I\|_{([W]_2^\bullet\cap[M]_0)(\inf J,\ta_\I)},}
for any $S'\in(\inf J,\ta_\I)$, and also on $(\ta_\I,\sup J)$ by the time symmetry. 
Thus we obtain \eqref{global St conv} and \eqref{local H conv} for any $\ta_\I\in[-\I,\I]$. 
\end{proof}

\section*{Acknowledgments}
The authors are grateful to Takahisa Inui, Tristan Roy, and Federica Sani for pointing out the errors.

\end{document}